\documentclass[10pt,a4paper,twoside]{amsart}

\usepackage{graphicx}
\usepackage{a4wide}

\newtheorem{definition}{Definition}[section]
\newtheorem{theorem}{Theorem}[section]
\newtheorem{lemma}{Lemma}[section]

\newtheorem*{maintheorem*}{Main Theorem}
\allowdisplaybreaks
\numberwithin{equation}{section}

\newcommand{\Kruzkov}{Kru\v{z}kov~}

\newcommand{\Set}[1]{\left\{#1\right\}}

\newcommand{\norm}[1]{\left\| #1 \right\|}

\newcommand{\epsk}{\eps_k}

\newcommand{\eps}{\varepsilon}

\newcommand{\ue}{u_\eps}

\newcommand{\uek}{u_{\eps_k}}

\newcommand{\Fe}{F_\eps}
\newcommand{\Pe}{P_\eps}

\newcommand{\Pek}{P_{\eps_k}}

\newcommand{\ges}{g_\eps}

\newcommand{\pt}{\partial_t}
\newcommand{\pphi}{\partial_{\phi}}
\newcommand{\pphiphi}{\partial_{\phi\phi}^2}

\newcommand{\pq}{\partial_{x_1}}

\newcommand{\px}{\partial_x }
\newcommand{\pxx}{\partial_{xx}^2}

\renewcommand{\i}{\ifmmode\mathit{\mathchar"7010 }\else\char"10 \fi}
\renewcommand{\j}{\ifmmode\mathit{\mathchar"7011 }\else\char"11 \fi}
\newcommand{\R}{\mathbb{R}}
\newcommand{\N}{\mathbb{N}}

\newcommand{\Hneg}{H_{\mathrm{loc}}^{-1}}
\newcommand{\CL}{\mathcal{L}}
\newcommand{\CLea}{\mathcal{L}}

\newcommand{\sgn}[1]{\mathrm{sign}\left(#1\right)}

{%

\begin{enumerate}}%
{\end{enumerate}}

{%

\begin{enumerate}}%
{\end{enumerate}}

\begin{document}\large

\title[The short pulse equation]{Wellposedness of bounded solutions \\of the non-homogeneous initial boundary\\ for the Short Pulse Equation}
\author[G. M. Coclite and L. di Ruvo]{Giuseppe Maria Coclite and Lorenzo di Ruvo}
\address[Giuseppe Maria Coclite and Lorenzo di Ruvo]
{\newline Department of Mathematics,   University of Bari, via E. Orabona 4, 70125 Bari,   Italy}
\email[]{giuseppemaria.coclite@uniba.it, lorenzo.diruvo@uniba.it}
\urladdr{http://www.dm.uniba.it/Members/coclitegm/}
\date{\today}



\begin{abstract}
The short pulse equation provides a model for the propagation of ultra-short light pulses in silica optical fibers. It is a nonlinear evolution equation.
In this paper  the wellposedness of bounded solutions for the inhomogeneous initial boundary value problem  associated to this equation is studied.
\end{abstract}

\maketitle


\section{Introduction}
\label{sec:intro}
The short pulse equation has the form
\begin{equation}
\label{eq:SPE}
2\pq\pphi A_{0}+\chi^{(3)}\pphiphi A_{0}^3 + \frac{1}{c_2^2}A_{0} =0,
\end{equation}
where $A_0$ is the light wave amplitude, $\phi= \frac{t-x}{\eps}$, $x_1= \eps x$, $\eps$ is a small scale parameter, and  $\chi^{(3)}$ is the third order magnetic susceptibility. \eqref{eq:SPE} was introduced recently by
Sch\"afer and Wayne \cite{SW} as  a model equation describing the propagation of ultra-short light pulses in silica optical fibers.
It provides also an approximation of nonlinear wave packets in dispersive media in the limit of few cycles on the ultra-short pulse scale.
Numerical simulations \cite{CJSW} show that the short pulse equation approximation
to Maxwell's equations in the case when the pulse spectrum is not narrowly localized
around the carrier frequency is better than the one obtained from the nonlinear Schr\"odinger
equation, which models the evolution of slowly varying wave trains.
Such ultra-short plays a key role in the development of future technologies of ultra-fast optical transmission of informations.

In \cite{B} the author studied a new hierarchy of equations containing the short pulse equation \eqref{eq:SPE} and the elastic
beam equation, which describes nonlinear transverse oscillations of elastic beams
under tension. He showed that the hierarchy of equations is integrable. He obtained the two compatible Hamiltonian structures
and constructed an infinite series of both local and nonlocal conserved charges. Moreover, he gave the Lax description for both systems.
The integrability and the existence of solitary wave solutions have been studied in \cite{SS, SS1}.

Well-posedness and wave breaking for the short pulse equation have been
studied in \cite{SW} and \cite{LPS}, respectively.

\cite{Bo} (table $4.1.2$, pag. $212$) shows that, for some polymers, $\chi^{(3)}$ is a negative constant. Therefore, \eqref{eq:SPE} reads
\begin{equation}
\label{eq:SPE2}
2\pq\pphi A_{0}-k^2\pphiphi A_{0}^3 + \frac{1}{c_2^2} A_{0} =0, \quad \chi^{(3)}=-k^2.
\end{equation}
Following \cite{AVB, dR, E, KCS}, we consider the admensional form of \eqref{eq:SPE2} 
\begin{equation}
\label{eq:SPE1}
\px\left( \pt u + 3u^2 \px u\right)=u.
\end{equation}
Indeed,  multiplying \eqref{eq:SPE2} by $-c_2^2$, we have
\begin{equation}
\label{eq:001}
-2c_2^2\pq\pphi A_{0}+c_2^2k^2\pphiphi A_{0}^3 = A_{0}
\end{equation}
Consider the following Robelo transformation (see \cite{AVB, E, KCS}):
\begin{equation}
\label{eq:change1}
x_1=D_1 t , \quad \phi = D_2 x ,
\end{equation}
where $D_1$ and $D_2$ are two constants that will be specified later. Therefore,
\begin{equation}
\label{eq:change2}
\pq = D_1\pt, \quad \pphi=D_2 \px. \end{equation}
Taking $A_{0}(x_1,\phi)=u(t,x)$, it follows from \eqref{eq:SPE} and \eqref{eq:change2} that
\begin{equation}
\label{eq:spe1}
-2c_2^2 D_1D_2\px(\pt u) + 3c_2^2k^2 D_2^2\px\left(u^2\px u\right)=u.
\end{equation}
We choose $D_1$, $D_2$ so that
\begin{equation*}
 2c_2^2 D_{1}D_{2}=-1,\qquad  c_2^2k^2 D_2^2=1,
\end{equation*}
that is
\begin{equation}
\label{eq:change3}
D_{1}=-\frac{k}{2c_2}, \quad D_{2}=\frac{1}{c_2 k}.
\end{equation}
Therefore, \eqref{eq:SPE1} follows from \eqref{eq:spe1} and \eqref{eq:change3}.

It is interesting to remind that equation \eqref{eq:SPE1} was proposed earlier in \cite{NSC} in the context of plasma physic.

We are interested in the initial-boundary value problem for this equation, so we augment \eqref{eq:SPE1} with the boundary condition
\begin{equation}
\label{eq:boundary}
u(t,0)=g(t), \qquad t>0,
\end{equation}
and the initial datum
\begin{equation}
\label{eq:init}
u(0,x)=u_0(x), \qquad x>0,
\end{equation}
on which we assume that
\begin{equation}
\label{eq:assinit}
u_0\in L^{\infty}(0,\infty)\cap L^{1}(0,\infty), \quad \int_{0}^{\infty} u_{0}(x) dx =0.
\end{equation}
On the function
\begin{equation}
\label{eq:def-di-P0}
P_{0}(x)=\int_{0}^{x}u_{0}(y)dy,
\end{equation}
we assume that
\begin{equation}
\label{eq:l-2-di-P0}
\norm{P_{0}}^2_{L^2(0,\infty)}=\int_{0}^{\infty}\left(\int_{0}^{x} u_{0}(y)dy\right)^2 dx < \infty.
\end{equation}
On the boundary datum $g$, we assume that
\begin{equation}
\label{eq:ass-g}
g(t)\in L^{\infty}(0,\infty).
\end{equation}

Integrating \eqref{eq:SPE1} in $(0,x)$ we gain the integro-differential formulation of  \eqref{eq:SPE1} (see \cite{SS})
\begin{equation}
\label{eq:OHw-u}
\begin{cases}
\pt u+3u^2\px u = \int^x_0 u(t,y) dy,&\qquad t>0, \ x>0,\\
u(t,0)=g(t),& \qquad t>0,\\
u(0,x)=u_0(x), &\qquad x>0,
\end{cases}
\end{equation}
that is equivalent to
\begin{equation}
\label{eq:integ}
\begin{cases}
\pt u+ 3 u^2\px u= P, &\qquad t>0, \ x>0,\\
\px P=u, &\qquad t>0, \ x>0,\\
u(t,0)=g(t),& \qquad t>0,\\
P(t,0)=0,& \qquad t>0,\\
u(0,x)=u_0(x), &\qquad x>0.
\end{cases}
\end{equation}

One of the main issues in the analysis of \eqref{eq:integ} is that the equation is not preserving the $L^1$ norm, 
as a consequence the nonlocal source term $P$ and the solution $u$ are a priori only locally bounded. 
Indeed, from \eqref{eq:OHw-u} and \eqref{eq:integ} is clear that we cannot have any $L^\infty$ bound without an $L^1$ bound.
Since we are interested in the bounded solutions of \eqref{eq:SPE1}, some assumptions on the decay at infinity of the initial condition $u_0$ are needed.
The unique useful conserved quantities are 
\begin{equation*}
t\longmapsto\int u(t,x)dx=0,\qquad t\longmapsto\int u^2(t,x)dx.
\end{equation*}
In the sense that if $u(t,\cdot)$ has zero mean at time $t=0$, then it will have zero mean at any time $t>0$.
In addition, the $L^2$ norm  of $u(t,\cdot)$ is constant with respect to $t$.
Therefore, we require that initial condition $u_0$ belongs to $L^2\cap L^\infty$ and has zero mean.

Due to the regularizing effect of the $P$ equation in \eqref{eq:integ} we have that
\begin{equation}
\label{eq:OHsmooth}
    u\in L^{\infty}((0,T)\times(0,\infty))\Longrightarrow P\in L^{\infty}(0,T;W^{1,\infty}(0,\infty)), \quad T>0.
\end{equation}
Therefore, if a map $u\in  L^{\infty}((0,T)\times(0,\infty)),\,T>0,$  satisfies, for every convex
map $\eta\in  C^2(\R)$,
\begin{equation}
\label{eq:OHentropy}
   \pt \eta(u)+ \px q(u)-\eta'(u) P\le 0, \qquad     q(u)=\int^u 3\xi^2 \eta'(\xi)\, d\xi,
\end{equation}
in the sense of distributions, then \cite[Theorem 1.1]{CKK}
provides the existence of strong trace $u^\tau_0$ on the
boundary $x=0$.

We give the following definition of solution (see \cite{BRN}):

\begin{definition}
\label{def:sol}
We say that  $u\in  L^{\infty}((0,T)\times(0,\infty))$, $T>0$, is an entropy solution of the initial-boundary
value problem \eqref{eq:SPE1}, \eqref{eq:boundary}, and  \eqref{eq:init} if
for every nonnegative test function $\phi\in C^2(\R^2)$ with compact support, and $c\in\R$
\begin{equation}
\label{eq:ent2}
\begin{split}
\int_{0}^{\infty}\!\!\!\!\int_{0}^{\infty}\Big(\vert u - c\vert\pt\phi&+\sgn{u-c}\left(u^3- c^3\right)\px\phi\Big)dtdx\\
&+\int_{0}^{\infty}\!\!\!\!\int_{0}^{\infty}\sgn{u-c}P\phi dtdx\\
&+\int_{0}^{\infty}\sgn{g(t)-c}\left((u^\tau_0(t))^3-c^3\right)\phi(t,0)dt\\
&+\int_{0}^{\infty}\vert u_{0}(x)-c\vert\phi(0,x)dx\geq 0,
\end{split}
\end{equation}
where $u^\tau_0(t)$ is the trace of $u$ on the boundary $x=0$.
\end{definition}

The main result of this paper  is the following theorem.
\begin{theorem}
\label{th:main}
Assume \eqref{eq:assinit}, \eqref{eq:l-2-di-P0} and \eqref{eq:ass-g}.
The initial-boundary value problem
\eqref{eq:SPE1}, \eqref{eq:boundary} and \eqref{eq:init} possesses
an unique entropy solution $u$ in the sense of Definition \ref{def:sol}.
Moreover, if $u$ and $v$ are two entropy solutions of  \eqref{eq:SPE1}, \eqref{eq:boundary}, \eqref{eq:init} in the sense of Definition \ref{def:sol} the following inequality holds
 \begin{equation}
 \label{eq:stability}
\norm{u(t,\cdot)-v(t,\cdot)}_{L^1(0,R)}\le  e^{C(T) t}\norm{u(0,\cdot)-v(0,\cdot)}_{L^1(0,R+C(T)t)},
\end{equation}
for almost every $0<t<T$, $R>0$, and some suitable constant $C(T)>0$.
\end{theorem}
The paper is organized as follows. In Section \ref{sec:vv} we prove several a priori estimates on a vanishing viscosity approximation of \eqref{eq:integ}.
Those play a key role in the proof of our main result, that is given in Section \ref{sec:proof}

\section{Vanishing viscosity approximation}
\label{sec:vv}
Our existence argument is based on passing to the limit
in a vanishing viscosity approximation of \eqref{eq:integ}.

Fix a small number $\eps>0$, and let $\ue=\ue (t,x)$ be the unique classical solution of the following mixed problem
\begin{equation}
\label{eq:OHepsw}
\begin{cases}
\pt \ue+3\ue^2\px \ue=\Pe+ \eps\pxx\ue,&\quad t>0,\ x>0,\\
\px\Pe=\ue,&\quad t>0,\ x>0,\\
\ue(t,0)=\ges(t),&\quad t>0,\\
\Pe(t,0)=0,&\quad t>0,\\
\ue(0,x)=u_{0,\eps}(x),&\quad x>0,
\end{cases}
\end{equation}
where $u_{\eps,0}$ and $\ges$ are $C^\infty(0,\infty)$ approximations of $u_{0}$ and $g$ such that
\begin{equation}
\begin{split}
\label{eq:u0eps}
& u_{0,\eps}\to u_{0},\quad \text{a.e. and in $L^p(0,\infty),\, 1\leq p< \infty,$},\\
&P_{0,\eps}\to P_{0},\quad \text{in $L^2(0,\infty)$},\\
&\ges\to g,\quad \text{a.e. and in $L^p_{loc}(0,\infty),\, 1\leq p< \infty$},\\
&\norm{u_{\eps,0}}_{L^{\infty}(0,\infty)}\le \norm{u_0}_{L^{\infty}(0,\infty)}, \quad \norm{u_{\eps,0}}_{L^2(0,\infty)}\le \norm{u_0}_{L^2(0,\infty)},\\
&\norm{u_{\eps,0}}_{L^{4}(0,\infty)}\le \norm{u_0}_{L^{4}(0,\infty)}, \quad  \int_{0}^{\infty} u_{\eps,0}(x) dx =0,\\
&\norm{P_{\eps,0}}_{L^2(0,\infty)}\le \norm{P_0}_{L^2(0,\infty)}, \quad \norm{\ges}_{L^{\infty}(0,\infty)}\le C_0,
\end{split}
\end{equation}
and $C_0$ is a constant independent on $\eps$.

Clearly, \eqref{eq:OHepsw} is equivalent to the integro-differential problem
\begin{equation}
\label{eq:OHepswint}
\begin{cases}
\pt \ue+3\ue^2\px\ue=\int_0^x \ue (t,y)dy+ \eps\pxx\ue,&\quad t>0,\ x>0 ,\\
\ue(t,0)=\ges(t),&\quad t>0,\\
\ue(0,x)=u_{\eps,0}(x),&\quad x>0.
\end{cases}
\end{equation}

Let us prove some a priori estimates on $\ue$ and $\Pe$, denoting with $C_0$ the constants which depend only on the initial data, and $C(T)$ the constants which depend also on $T$.

Arguing as \cite[Lemma $1$]{CdK}, or \cite[Lemma $2.2.1$]{dR}, we have the following result.
\begin{lemma}
\label{lm:cns}
The following statements are equivalent
\begin{align}
\label{eq:con-u}
\int_{0}^{\infty} \ue(t,x) dx &=0, \quad t\ge 0,\\
\label{eq:120}
\frac{d}{dt}\int_0^{\infty} \ue^2dx + 2\eps\int_{0}^{\infty}(\px\ue)^2 dx&=\frac{3}{2}\ges^4(t)+2\eps\ges(t)\px\ue(t,0),\quad t>0.
\end{align}
\end{lemma}
\begin{proof}
Let $t>0$. We begin by proving that \eqref{eq:con-u} implies \eqref{eq:120}. Multiplying \eqref{eq:OHepswint} by $\ue$, an integration on $(0,\infty)$ gives
\begin{align*}
\frac{d}{dt}\int_{0}^{\infty}\ue^2 dx =& 2\int_{0}^{\infty}\ue\pt\ue dx\\
=& 2\eps\int_{0}^{\infty}\ue\pxx\ue dx - 6\int_{0}^{\infty}\ue^3 \px\ue dx+2\int_{0}^{\infty}\ue \left(\int_{0}^x\ue dy \right)dx\\
=&2\eps\px\ue(t,0)\ges(t) - 2\eps\int_{0}^{\infty}(\px\ue)^2 dx +\frac{3}{2}\ges^4(t)+2\int_{0}^{\infty}\ue \left(\int_{0}^x\ue dy \right)dx.
\end{align*}
By \eqref{eq:OHepsw},
\begin{equation*}
2\int_{0}^{\infty}\ue \left(\int_{0}^x\ue dy \right)dx= 2\int_{0}^{\infty}\Pe\px\Pe dx=\Pe^2(t,\infty).
\end{equation*}
Then,
\begin{equation}
\label{eq:U12}
\frac{d}{dt}\int_{0}^{\infty}\ue^2 dx +2\eps\int_{0}^{\infty}(\px\ue)^2 dx= \Pe^2(t,\infty)+2\eps\px\ue(t,0)\ges(t) +\frac{3}{2}\ges^4(t).
\end{equation}
Thanks to \eqref{eq:con-u},
\begin{equation}
\label{eq:U13}
\lim_{x\to\infty}\Pe^2(t,x) =\left(\int_{0}^{\infty} \ue(t,x) dx\right)^2=0.
\end{equation}
\eqref{eq:U12} and \eqref{eq:U13} give \eqref{eq:120}.

Let us show that \eqref{eq:120} implies \eqref{eq:con-u}. We assume by contradiction that \eqref{eq:con-u} does not hold, namely:
\begin{equation*}
\int_{0}^{\infty}\ue(t,x) dx \neq 0.
\end{equation*}
For \eqref{eq:integ},
\begin{equation*}
\Pe^2(t,\infty)= \left(\int_{0}^{\infty}\ue(t,x) dx\right)^2 \neq 0.
\end{equation*}
Therefore, \eqref{eq:U12} gives
\begin{equation*}
\frac{d}{dt}\int_{0}^{\infty}\ue^2 dx +2\eps\int_{0}^{\infty}(\px\ue)^2 dx\neq 2\eps\px\ue(t,0)\ges(t) +\frac{3}{2}\ges^4(t),
\end{equation*}
which is in contradiction with \eqref{eq:120}.
\end{proof}
\begin{lemma}
\label{lm:l2-u}
For each $t \ge 0$, \eqref{eq:con-u} holds true.
In particular, we have that
\begin{equation}
\label{eq:l2-u}
\norm{\ue(t,\cdot)}_{L^2(0,\infty)}^2+2\eps \int_0^t \norm{\px \ue(s,\cdot)}^2_{L^2(0,\infty)}ds\le C_{0}(t+1) +2\eps\int_{0}^{t}\ges(t)\px\ue(t,0) ds.
\end{equation}
\end{lemma}
\begin{proof}
We begin by observing that $\pt\ue(t,0)=\ges'(t)$, being $\ue(t,0)=\ges(t)$. It follows from \eqref{eq:OHepswint} that
\begin{equation}
\begin{split}
\label{eq:125}
\eps\pxx\ue(t,0)=& \pt\ue(t,0) + 3\ue^2(t,0)\px\ue(t,0)-\int_{0}^{0}\ue(t,x)dx\\
=& \ges'(t) + 3\ges^2(t)\px\ue(t,0).
\end{split}
\end{equation}
Differentiating \eqref{eq:OHepswint} with respect to $x$, we have
\begin{equation*}
\px(\pt\ue + 3\ue^2\px\ue - \eps \pxx \ue)=\ue.
\end{equation*}
For \eqref{eq:125}, and being $\ue$ a smooth solution of \eqref{eq:OHepswint}, an integration over $(0,\infty)$ gives \eqref{eq:con-u}.
Lemma \ref{lm:cns} says that also \eqref{eq:120}  holds true.
Therefore, integrating \eqref{eq:120} on $(0,t)$, for \eqref{eq:u0eps}, we have
\begin{align*}
\norm{\ue(t,\cdot)}_{L^2(0,\infty)}^2&+2\eps \int_0^t \norm{\px \ue(s,\cdot)}^2_{L^2(0,\infty)}ds\\
\le& \norm{u_0}^2_{L^2(0,\infty)}+ \frac{3}{2}\int_{0}^{t}\ges^4(s)ds+2\eps\int_{0}^{t}\ges(s)\px\ue(s,0)ds\\
\le& \norm{u_0}^2_{L^2(0,\infty)}+ \frac{3}{2}\norm{\ges}^4_{L^{\infty}(0,\infty)}t+ 2\eps\int_{0}^{t}\ges(s)\px\ue(s,0)ds\\
\le & \norm{u_0}^2_{L^2(0,\infty)}+C_{0}t + 2\eps\int_{0}^{t}\ges(s)\px\ue(s,0)ds,
\end{align*}
which gives \eqref{eq:l2-u}.
\end{proof}

\begin{lemma}
We have that
\begin{equation}
\label{eq:lim-infty-f}
\lim_{x\to\infty}\Fe(t,x)=\int_{0}^{\infty}\Pe(t,x)dx= \eps\px\ue(t,0)-\ges^3(t),
\end{equation}
where
\begin{equation}
\label{eq:def-F}
\Fe(t,x)=\int_{0}^{x}\Pe(t,y)dy.
\end{equation}
\end{lemma}
\begin{proof}
We begin by observing that, integrating on $(0,x)$ the second equation of \eqref{eq:OHepsw}, we get
\begin{equation}
\label{eq:p1}
\Pe(t,x)= \int_{0}^{x}\ue(t,y)dy.
\end{equation}
Differentiating \eqref{eq:p1} with respect to $t$, we have
\begin{equation}
\label{eq:p2}
\pt\Pe(t,x)=\int_{0}^{x}\pt\ue(t,y)dy=\frac{d}{dt}\int_{0}^{x}\ue(t,y)dy.
\end{equation}
It follows from \eqref{eq:con-u} and \eqref{eq:p2} that
\begin{equation}
\label{eq:p3}
\lim_{x\to\infty}\pt\Pe(t,x)=\frac{d}{dt}\int_{0}^{\infty}\ue(t,x)dx=0.
\end{equation}
Integrating on $(0,x)$ the first equation of \eqref{eq:OHepsw}, thanks to \eqref{eq:p2}, we have
\begin{equation}
\label{eq:equa-in-P}
\pt\Pe(t,x) +\ue^3(t,x)- \ges^3(t)-\eps\px\ue(t,x) + \eps\px\ue(t,0)=\int_{0}^{x}\Pe(t,y)dy.
\end{equation}
It follows from the regularity of $\ue$ that
\begin{equation}
\label{eq:lim}
\lim_{x\to\infty}\left(\ue^3(t,x))-\eps\px\ue(t,x)\right)=0.
\end{equation}
\eqref{eq:p3} and \eqref{eq:lim} give \eqref{eq:lim-infty-f}.
\end{proof}
Arguing as in \cite[Lemma $2.3$]{Cd2}, we prove the following lemma.
\begin{lemma}
\label{lm:P-2u-4}
Let $T>0$. There exists a constant $C(T)>0$, independent on $\eps$, such that
\begin{equation}
\label{eq:u-4-p-2}
\begin{split}
\norm{\ue(t,\cdot)}^4_{L^{4}(0,\infty)}&+2\norm{\Pe(t,\cdot)}^2_{L^2(0,\infty)}\\
&+12\eps\int_{0}^{t}\norm{\ue(s,\cdot)\px\ue(s,\cdot)}^2_{L^2(0,\infty)}ds\\
&+4\eps\int_{0}^{t}\norm{\ue(s,\cdot)}_{L^2(0,\infty)}^2 ds +\eps^2\int_{0}^{t}\left(\px\ue(s,0)\right)^2 ds \le C(T),
\end{split}
\end{equation}
for every $0\le t\le T$.
\end{lemma}
\begin{proof}
Let $0\le t\le T$. We begin by observing that \eqref{eq:def-F} and \eqref{eq:equa-in-P} imply
\begin{equation}
\label{eq:p5}
\pt\Pe(t,x)= \Fe(t,x)-\ue^3(t,x)+\ges^3(t)+\eps\ue(t,x)-\eps\px\ue(t,0).
\end{equation}
Multiplying \eqref{eq:p5} by $\Pe$, an integration on $(0,\infty)$ gives
\begin{equation}
\label{eq:p8}
\begin{split}
\frac{d}{dt}\int_{0}^{\infty}\Pe^2 dx=&2\int_{0}^{\infty}\Pe\pt\Pe dx\\
=&2\int_{0}^{\infty}\Pe\Fe dx -2\int_{0}^{\infty}\ue^3\Pe dx+2\ges^3(t)\int_{0}^{\infty}\Pe dx\\
&+2\eps\int_{0}^{\infty}\px\ue\Pe dx -2\eps\px\ue(t,0)\int_{0}^{\infty}\Pe dx.
\end{split}
\end{equation}
By \eqref{eq:OHepsw},
\begin{equation}
\label{eq:p9}
2\int_{0}^{\infty}\px\ue\Pe dx = -2\eps\int_{0}^{\infty}\ue\px\Pe dx= -2\eps\norm{\ue(t,\cdot)}_{L^2(0,\infty)}^2
\end{equation}
while, in light of \eqref{eq:def-F} and \eqref{eq:lim-infty-f},
\begin{equation}
\label{eq:10}
\begin{split}
2\int_{0}^{\infty}\Pe\Fe dx =& 2\int_{0}^{\infty}\Fe\px\Fe dx \\
=&\Fe^2(t,\infty)= \left(\eps\px\ue(t,0)-\ges^3(t)\right)^2\\
=&\eps^2\left(\px\ue(t,0)\right)^2- 2\eps\px\ue(t,0)\ges^3(t)+ \ges^6(t).
\end{split}
\end{equation}
Using again \eqref{eq:lim-infty-f},
\begin{equation}
\label{eq:p11}
\begin{split}
-2\eps\px\ue(t,0)\int_{0}^{\infty}\Pe dx=& -2\eps^2\left(\px\ue(t,0)\right)^2 + 2\eps\px\ue(t,0)\ges^3(t),\\
2\ges^3(t)\int_{0}^{\infty}\Pe dx=& 2\eps\px\ue(t,0)\ges^3(t)-2\ges^6(t).
\end{split}
\end{equation}
\eqref{eq:p8}, \eqref{eq:p9}, \eqref{eq:10} and \eqref{eq:p11} give
\begin{align*}
\frac{d}{dt}\norm{\Pe(t,\cdot)}^2_{L^2(0,\infty)}=& -2\eps\norm{\ue(t,\cdot)}_{L^2(0,\infty)}^2 -2\int_{0}^{\infty}\ue^3\Pe dx\\
&-\eps^2\left(\px\ue(t,0)\right)^2 - \ges^6(t)+ 2\eps\px\ue(t,0)\ges^3(t),
\end{align*}
that is,
\begin{equation}
\label{eq:p15}
\begin{split}
\frac{d}{dt}\norm{\Pe(t,\cdot)}^2_{L^2(0,\infty)}&+2\eps\norm{\ue(t,\cdot)}_{L^2(0,\infty)}^2+\eps^2\left(\px\ue(t,0)\right)^2\\
=& -2\int_{0}^{\infty}\ue^3\Pe dx - \ges^6(t)+ 2\eps\px\ue(t,0)\ges^3(t).
\end{split}
\end{equation}
Multiplying \eqref{eq:OHepsw} by $2\ue^3$, an integration on $(0,\infty)$ gives
\begin{align*}
\frac{d}{dt}\left(\frac{1}{2}\int_{0}^{\infty}\ue^4 dx\right)=&2\int_{0}^{\infty}\ue^3\pt\ue dx\\
=&-6\int_{0}^{\infty}\ue^5 \px\ue dx +2\int_{0}^{\infty}\ue^3\Pe dx + 2\eps\int_{0}^{\infty}\ue^3\pxx\ue dx\\
=&\ges^6(t)+2\int_{0}^{\infty}\ue^3\Pe dx+ 2 \eps\px\ue(t,0)\ges^3(t)-6\eps\int_{0}^{\infty}\ue^2(\px\ue)^2dx,
\end{align*}
that is
\begin{equation}
\label{eq:u1}
\begin{split}
\frac{d}{dt}\left(\frac{1}{2}\norm{\ue(t,\cdot)}^4_{L^{4}(0,\infty)}\right)&+ 6\eps\norm{\ue(t,\cdot)\px\ue(t,\cdot)}^2_{L^2(0,\infty)}\\
&=\ges^6(t)+2\int_{0}^{\infty}\ue^3\Pe dx+ 2 \eps\px\ue(t,0)\ges^3(t).
\end{split}
\end{equation}
Adding \eqref{eq:p15} and \eqref{eq:u1}, we get
\begin{equation}
\label{eq:u2}
\begin{split}
&\frac{d}{dt}\left(\frac{1}{2}\norm{\ue(t,\cdot)}^4_{L^{4}(0,\infty)}+ \norm{\Pe(t,\cdot)}^2_{L^2(0,\infty)}\right)\\
&\quad + 6\eps\norm{\ue(t,\cdot)\px\ue(t,\cdot)}^2_{L^2(0,\infty)}+ 2\eps\norm{\ue(t,\cdot)}_{L^2(0,\infty)}^2\\
&\quad +\eps^2\left(\px\ue(t,0)\right)^2 = 4\eps\px\ue(t,0)\ges^3(t).
\end{split}
\end{equation}
Due to the Young inequality,
\begin{equation}
\label{eq:you1}
4\eps\px\ue(t,0)\ges^3(t)\le\left\vert\eps\px\ue(t,0)\right\vert\left\vert 4\ges^3(t)\right\vert\le \frac{\eps^2}{2}\left(\px\ue(t,0)\right)^2+8\ges^6(t).
\end{equation}
It follows from \eqref{eq:u2} and  \eqref{eq:you1} that
\begin{equation}
\label{eq:u3}
\begin{split}
&\frac{d}{dt}\left(\frac{1}{2}\norm{\ue(t,\cdot)}^4_{L^{4}(0,\infty)}+ \norm{\Pe(t,\cdot)}^2_{L^2(0,\infty)}\right)\\
&\quad + 6\eps\norm{\ue(t,\cdot)\px\ue(t,\cdot)}^2_{L^2(0,\infty)}+ 2\eps\norm{\ue(t,\cdot)}_{L^2(0,\infty)}^2\\
&\quad +\frac{\eps^2}{2}\left(\px\ue(t,0)\right)^2 \le 8\ges^6(t).
\end{split}
\end{equation}
Integrating \eqref{eq:u3} on $(0,t)$, by \eqref{eq:u0eps}, we have
\begin{align*}
&\frac{1}{2}\norm{\ue(t,\cdot)}^4_{L^{4}(0,\infty)}+ \norm{\Pe(t,\cdot)}^2_{L^2(0,\infty)}+6\eps\int_{0}^{t}\norm{\ue(s,\cdot)\px\ue(s,\cdot)}^2_{L^2(0,\infty)}ds\\
&\quad +2\eps\int_{0}^{t}\norm{\ue(s,\cdot)}_{L^2(0,\infty)}^2 ds + \frac{\eps^2}{2}\int_{0}^{t}\left(\px\ue(s,0)\right)^2 ds \\
&\qquad \le  \norm{u_0}^4_{L^{4}(0,\infty)}+ \norm{P_0}^2_{L^2(0,\infty)}+8\int_{0}^{t}\ges^6(s)ds\\
&\qquad \le C_{0} + 8\norm{\ges}^6_{L^{\infty}(0,\infty)}t\le C_{0}\left(1+8t\right),
\end{align*}
which gives \eqref{eq:u-4-p-2}.
\end{proof}
\begin{lemma}
\label{lm:l2-u-1}
Let $T>0$. There exists a constant $C(T)>0$, independent on $\eps$, such that
\begin{equation}
\label{eq:l2-u-1}
\norm{\ue(t,\cdot)}_{L^2(0,\infty)}^2+2\eps \int_0^t \norm{\px \ue(s,\cdot)}^2_{L^2(0,\infty)}ds\le C(T),
\end{equation}
for every $0\le t\le T$.
In particular, we have
\begin{equation}
\label{eq:p-infty}
\norm{\Pe}_{L^{\infty}((0,T)\times(0,\infty))}\le C(T).
\end{equation}
\end{lemma}
\begin{proof}
We begin by observing that, using the Young inequality,
\begin{equation*}
2\eps\ges(t)\px\ue(t,0)\le 2 \left \vert \ges(t) \right\vert\left\vert \eps \px\ue(t,0)\right\vert \le \ges^2(t) + \eps^2 \left(\px\ue(t,0)\right)^2.
\end{equation*}
Therefore, in light of  \eqref{eq:u0eps} and \eqref{eq:u-4-p-2},
\begin{equation}
\label{eq:you2}
\begin{split}
2\eps\int_{0}^{t}\ges(s)\px\ue(s,0)ds\le& 2\int_{0}^{t}\left\vert\ges(t)\right\vert\left\vert\eps\px\ue(t,0)\right\vert dx\\
\le & \int_{0}^{t}\ges^2(s)ds + \eps^2 \int_{0}^{t}\left(\px\ue(s,0)\right)^2 ds\\
\le & \norm{\ges}^2_{L^{\infty}(0,\infty)} t + \eps^2 \int_{0}^{t}\left(\px\ue(s,0)\right)^2 ds\\
\le & C_{0}t  + \eps^2 \int_{0}^{t}\left(\px\ue(s,0)\right)^2 ds \le C(T).
\end{split}
\end{equation}
\eqref{eq:l2-u-1} follows from \eqref{eq:l2-u} and \eqref{eq:you2}.

Finally, we prove \eqref{eq:p-infty}. Due to \eqref{eq:OHepsw}, \eqref{eq:u-4-p-2}, \eqref{eq:l2-u-1} and the H\"older inequality,
\begin{align*}
\Pe^2(t,x)=2\int_{0}^{x}\Pe\px\Pe dy \le& 2 \int_{0}^{\infty}\vert \Pe\vert \vert\px\Pe\vert dx\\
\le& 2\norm{\Pe(t,\cdot)}_{L^2(0,\infty)}\norm{\px\Pe(t,\cdot)}_{L^2(0,\infty)}\\
=&2 \norm{\Pe(t,\cdot)}_{L^2(0,\infty)}\norm{\ue(t,\cdot)}_{L^2(0,\infty)}\le C(T).
\end{align*}
Therefore,
\begin{equation*}
\vert \Pe(t,x)\vert \le C(T),
\end{equation*}
which gives \eqref{eq:p-infty}.
\end{proof}
\begin{lemma}
\label{lm:linfty-u}
Let $T>0$. We have 
\begin{equation}
\label{eq:linfty-u}
\norm{\ue}_{L^\infty((0,T)\times(0,\infty))}\le\norm{u_0}_{L^\infty(0,\infty)}+C(T).
\end{equation}
\end{lemma}
\begin{proof}
Due to \eqref{eq:OHepsw} and \eqref{eq:p-infty},
\begin{equation*}
\pt \ue +3\ue^2\px\ue-\eps\pxx \ue\le C(T).
\end{equation*}
Since the map
\begin{equation*}
{\mathcal F}(t):=\norm{u_0}_{L^\infty(0,\infty)}+ C(T)t,
\end{equation*}
solves the equation
\begin{equation*}
\frac{d{\mathcal F}}{dt}= C(T)
\end{equation*}
and
\begin{equation*}
\max\{\ue(0,x),0\}\le {\mathcal F}(t),\qquad (t,x)\in (0,T)\times (0,\infty),
\end{equation*}
the comparison principle for parabolic equations implies that
\begin{equation*}
 \ue(t,x)\le {\mathcal F}(t),\qquad (t,x)\in (0,T)\times (0,\infty).
\end{equation*}
In a similar way we can prove that
\begin{equation*}
\ue(t,x)\ge -{\mathcal F}(t),\qquad (t,x)\in (0,T)\times (0,\infty).
\end{equation*}
Therefore,
\begin{equation*}
\vert\ue(t,x)\vert\le\norm{u_0}_{L^\infty(0,\infty)}+ C(T)t\le  \norm{u_0}_{L^\infty(0,\infty)}+ C(T)T,
\end{equation*}
which gives \eqref{eq:linfty-u}.
\end{proof}

\section{Proof of Theorem \ref{th:main}}
\label{sec:proof}
This section is devoted to the proof of Theorem \ref{th:main}.

Let us begin by proving the existence of  a distributional solution
to  \eqref{eq:SPE1}, \eqref{eq:boundary}, \eqref{eq:init}  satisfying \eqref{eq:ent2}.
\begin{lemma}\label{lm:conv}
Let $T>0$. There exists a function $u\in L^{\infty}((0,T)\times (0,\infty))$ that is a distributional
solution of \eqref{eq:integ} and satisfies  \eqref{eq:ent2}.
\end{lemma}
We  construct a solution by passing
to the limit in a sequence $\Set{u_{\eps}}_{\eps>0}$ of viscosity
approximations \eqref{eq:OHepsw}. We use the
compensated compactness method \cite{TartarI}.

\begin{lemma}\label{lm:conv-u}
Let $T>0$. There exists a subsequence
$\{\uek\}_{k\in\N}$ of $\{\ue\}_{\eps>0}$
and a limit function $  u\in L^{\infty}((0,T)\times(0,\infty))$
such that
\begin{equation}\label{eq:convu}
    \textrm{$\uek \to u$ a.e.~and in $L^{p}_{loc}((0,T)\times(0,\infty))$, $1\le p<\infty$}.
\end{equation}
Moreover, we have
\begin{equation}
\label{eq:conv-P}
\textrm{$\Pek \to P$ a.e.~and in $L^{p}_{loc}(0,T;W^{1,p}_{loc}(0,\infty))$, $1\le p<\infty$},
\end{equation}
where
\begin{equation}
\label{eq:tildeu}
P(t,x)=\int_0^x u(t,y)dy,\qquad t\ge 0,\quad x\ge 0,
\end{equation}
and \eqref{eq:ent2} holds true.
\end{lemma}

\begin{proof}
Let $\eta:\R\to\R$ be any convex $C^2$ entropy function, and
$q:\R\to\R$ be the corresponding entropy
flux defined by $q'(u)=3u^2\eta'(u)$.
By multiplying the first equation in \eqref{eq:OHepsw} with
$\eta'(\ue)$ and using the chain rule, we get
\begin{equation*}
    \pt  \eta(\ue)+\px q(\ue)
    =\underbrace{\eps \pxx \eta(\ue)}_{=:\CLea_{1,\eps}}
    \, \underbrace{-\eps \eta''(\ue)\left(\px  \ue\right)^2}_{=: \CLea_{2,\eps}}
     \, \underbrace{+\eta'(\ue) \Pe}_{=: \CLea_{3,\eps}},
\end{equation*}
where  $\CLea_{1,\eps}$, $\CLea_{2,\eps}$, $\CLea_{3,\eps}$ are distributions.
Let us show that
\begin{equation*}
\label{eq:H1}
\textrm{$\CLea_{1,\eps}\to 0$ in $H^{-1}((0,T)\times(0,\infty))$, $T>0$.}
\end{equation*}
Since
\begin{equation*}
\eps\pxx\eta(\ue)=\px(\eps\eta'(\ue)\px\ue),
\end{equation*}
for \eqref{eq:l2-u-1} and Lemma \ref{lm:linfty-u},
\begin{align*}
\norm{\eps\eta'(\ue)\px\ue}^2_{L^2((0,T)\times (0,\infty))}&\le\eps ^2\norm{\eta'}^2_{L^{\infty}(J_T)}\int_{0}^{T}\norm{\px\ue(s,\cdot)}^2_{L^2(0,\infty)}ds\\
&\le\eps\norm{\eta'}^2_{L^{\infty}(J_T)}C(T)\to 0,
\end{align*}
where
\begin{equation*}
J_T=\left(-\norm{u_0}_{L^\infty(0,\infty)}- C(T), \norm{u_0}_{L^\infty(0,\infty)}+C(T)\right).
\end{equation*}
We claim that
\begin{equation*}
\label{eq:L1}
\textrm{$\{\CLea_{2,\eps}\}_{\eps>0}$ is uniformly bounded in $L^1((0,T)\times(0,\infty))$, $T>0$}.
\end{equation*}
Again by \eqref{eq:l2-u-1} and Lemma \ref{lm:linfty-u},
\begin{align*}
\norm{\eps\eta''(\ue)(\px\ue)^2}_{L^1((0,T)\times (0,\infty))}&\le
\norm{\eta''}_{L^{\infty}(J_T)}\eps
\int_{0}^{T}\norm{\px\ue(s,\cdot)}^2_{L^2(0,\infty)}ds\\
&\le \norm{\eta''}_{L^{\infty}(J_T)}C(T).
\end{align*}
We have that
\begin{equation*}
\textrm{$\{\CL_{3,\eps}\}_{\eps>0}$ is uniformly bounded in $L^1_{loc}((0,T)\times (0,\infty))$, $T>0$.}
\end{equation*}
Let $K$ be a compact subset of $(0,T)\times (0,\infty)$. Using \eqref{eq:p-infty} and Lemma \ref{lm:linfty-u},
\begin{align*}
\norm{\eta'(\ue)\Pe}_{L^1(K)}&=\int_{K}\vert\eta'(\ue)\vert\vert\Pe\vert
dtdx\\
&\leq 
\norm{\eta'}_{L^{\infty}(J_T)}\norm{\Pe}_{L^{\infty}(I_{T})}\vert K \vert .
\end{align*}
Therefore, Murat's lemma \cite{Murat:Hneg} implies that
\begin{equation}
\label{eq:GMC1}
    \text{$\left\{  \pt  \eta(\ue)+\px q(\ue)\right\}_{\eps>0}$
    lies in a compact subset of $\Hneg((0,T)\times(0,\infty))$.}
\end{equation}
The $L^{\infty}$ bound stated in Lemma \ref{lm:linfty-u}, \eqref{eq:GMC1}, and the
 Tartar's compensated compactness method \cite{TartarI} give the existence of a subsequence
$\{\uek\}_{k\in\N}$ and a limit function $  u\in L^{\infty}((0,T)\times(0,\infty)),\,T>0,$
such that \eqref{eq:convu} holds.

\eqref{eq:conv-P} follows from \eqref{eq:convu}, the H\"older inequality and the identity
\begin{equation*}
P_{\eps_{k}}=\int_{0}^{x}u_{\eps_{k}} dy, \quad \px P_{\eps_{k}}=u_{\eps_{k}}.
\end{equation*}
Finally, we prove \eqref{eq:ent2}.\\
Let $k\in\N$, $c\in\R$ be a constant, and $\phi\in C^{\infty}(\R^2)$ be a nonnegative test function with compact support. Multiplying the first equation of\eqref{eq:OHepsw} by $\sgn{\ue-c}$, we have
\begin{align*}
\pt \vert \uek -c\vert&+\px\left(\sgn{\uek-c}\left(\uek^3-c^3\right)\right)\\
&-\sgn{\uek-c}\Pek-\epsk\pxx\vert \uek -c\vert \le 0.
\end{align*}
Multiplying by $\phi$ and integrating over $(0,\infty)^2$, we get
\begin{align*}
&\int_{0}^{\infty}\!\!\!\!\int_{0}^{\infty}\left(\vert \uek -c\vert \pt\phi+\left(\sgn{\uek-c}\left(\uek^3-c^3\right)\right)\px\phi\right)dtdx\\
&\qquad +\int_{0}^{\infty}\!\!\!\!\int_{0}^{\infty}\sgn{\uek-c}\Pek dtdx -\epsk  \int_{0}^{\infty}\!\!\!\!\int_{0}^{\infty}\px\vert \uek -c\vert\px\phi dtdx\\
&\qquad + \int_{0}^{\infty}\vert u_{0}(x) -c\vert\phi(0,x) dx + \int_{0}^{\infty}\sgn{g_{\eps_{k}}(t)-c}\left(g^3_{\eps_{k}}(t)-c^3\right)\phi(t,0)dt\\
&\qquad -\epsk\int_{0}^{\infty}\px\vert \uek(t,0) -c\vert\phi(t,0)dt \ge 0.
\end{align*}
Thanks to \eqref{eq:u0eps} and  Lemmas \ref{lm:l2-u-1} and \ref{lm:linfty-u}, when $k\to \infty$, we have
\begin{align*}
&\int_{0}^{\infty}\!\!\!\!\int_{0}^{\infty}\left(\vert u -c\vert \pt\phi+\left(\sgn{u-c}\left(u^3-c^3\right)\right)\px\phi\right)dtdx\\
&\qquad +\int_{0}^{\infty}\!\!\!\!\int_{0}^{\infty}\sgn{u-c}P dtdx + \int_{0}^{\infty}\vert u_{0}(x) -c\vert\phi(0,x) dx\\
&\qquad +\int_{0}^{\infty}\sgn{g(t)-c}\left(g^3(t)-c^3\right)\phi(t,0)dt\\
&\qquad -\lim_{\epsk}\epsk\int_{0}^{\infty}\px\vert \uek(t,0) -c\vert\phi(t,0)dt \ge 0.
\end{align*}
We have to prove that (see \cite{BRN})
\begin{equation}
\label{eq:ux-in-0}
\begin{split}
&\lim_{\epsk}\epsk\int_{0}^{\infty}\px\vert \uek(t,0) -c\vert\phi(t,0)dt\\
&\qquad = \int_{0}^{\infty}\sgn{g(t)-c}\left(g^3(t)-(u^\tau_0(t))^3\right)\phi(t,0)dt.
\end{split}
\end{equation}
Let $\{\rho_{\nu}\}_{\nu\in\N}\subset C^{\infty}(\R)$ be such that
\begin{equation}
0\le \rho_{\nu} \le 1, \quad \rho_{\nu}(0)=1, \quad \vert \rho'_{\nu}\vert \le 1, \quad x\ge \frac{1}{\nu} \quad \Longrightarrow \quad \rho_{\nu}(x)=0.
\end{equation}
Using $(t,x) \mapsto \rho_{\nu}(x)\phi(t,x)$ as test function for the first equation of \eqref{eq:OHepsw} we get
\begin{align*}
&\int_{0}^{\infty}\!\!\!\!\int_{0}^{\infty}\left(\uek\pt\phi\rho_{\nu}+\uek^3\px\phi\rho_{\nu} + \uek^3\phi\rho'_{\nu}\right)dtdx +\int_{0}^{\infty}\!\!\!\!\int_{0}^{\infty}\Pek\phi\rho_{\nu} dtdx\\
&\qquad -\epsk\int_{0}^{\infty}\!\!\!\!\int_{0}^{\infty}\px\uek\left(\px\phi\rho_{\nu}+\phi\rho'_{\nu}\right)dtdx + \int_{0}^{\infty} u_{0}(x)\phi(0,x)\rho_{\nu}(x)dx\\
&\qquad +\int_{0}^{\infty} g^3_{\epsk}(t)\phi(t,0) dt -\epsk\int_{0}^{\infty}\px\uek(t,0)\phi(t,0) dt=0.
\end{align*}
As $k\to\infty$, we obtain that
\begin{align*}
&\int_{0}^{\infty}\!\!\!\!\int_{0}^{\infty}\left(u\pt\phi\rho_{\nu}+u^3\px\phi\rho_{\nu} + u^3\phi\rho'_{\nu}\right)dtdx +\int_{0}^{\infty}\!\!\!\!\int_{0}^{\infty}P\phi\rho_{\nu}dtdx\\
&\qquad + \int_{0}^{\infty} u_{0}(x)\phi(0,x)\rho_{\nu}dx +\int_{0}^{\infty} g^3(t)\phi(t,0) dt\\
&\quad  =\lim_{\epsk}\epsk\int_{0}^{\infty}\px\uek(t,0)\phi(t,0) dt.
\end{align*}
Sending $\nu\to\infty$, we get
\begin{equation*}
\lim_{\epsk}\epsk\int_{0}^{\infty}\px\uek(t,0)\phi(t,0) dt= \int_{0}^{\infty}\left(g^3(t)- (u^\tau_0(t))^3\right)\phi(t,0) dt.
\end{equation*}
Therefore, due to the strong convergence of $g_{\epsk}$ and the continuity of $g$ we have
\begin{align*}
&\lim_{\epsk}\epsk\int_{0}^{\infty}\px\vert \uek(t,0)-c\vert\phi(t,0)dt\\
&\qquad = \lim_{\epsk} \int_{0}^{\infty}\px\uek(t,0)\sgn{\uek(t,0)-c}\phi(t,0)dt\\
&\qquad = \lim_{\epsk} \int_{0}^{\infty}\px\uek(t,0)\sgn{g_{\epsk}(t)-c}\phi(t,0)dt\\
&\qquad = \int_{0}^{\infty}\sgn{g(t)-c}\left(g^3(t)- (u^\tau_0(t))^3\right)\phi(t,0) dt,
\end{align*}
that is \eqref{eq:ux-in-0}.
\end{proof}

\begin{proof}[Proof of Theorem \ref{th:main}]
Lemma \eqref{lm:conv-u} gives the existence of an entropy solution $u$ for \eqref{eq:OHw-u}, or
equivalently \eqref{eq:integ}.

We observe that, fixed $T>0$, the solutions of \eqref{eq:OHw-u}, or equivalently \eqref{eq:integ}, are bounded in $(0,T)\times\R$.
Therefore, using \cite[Theorem $1.1$]{Cd}, $u$ is unique and \eqref{eq:stability} holds true.
\end{proof}

\end{document}